\documentclass[12pt]{amsart}
\usepackage{amscd,amssymb,amsmath,multicol,float}
\restylefloat{table}
\newtheorem{thm}[equation]{Theorem}
\numberwithin{equation}{section}
\newtheorem{cor}[equation]{Corollary}

\newtheorem{rmk}[equation]{Remark}

\newtheorem{claim}[equation]{Claim}
\newtheorem{lem}[equation]{Lemma}

\newtheorem{conj}[equation]{Conjecture}
\newtheorem{defin}[equation]{Definition}
\newtheorem{diag}[equation]{Table}
\newtheorem{prop}[equation]{Proposition}

\renewcommand{\theequation}{\thesection.\arabic{equation}}
\begin{document}
\raggedbottom \voffset=-.7truein \hoffset=0truein \vsize=8truein
\hsize=6truein \textheight=8truein \textwidth=6truein
\baselineskip=18truept

\def\mapright#1{\ \smash{\mathop{\longrightarrow}\limits^{#1}}\ }
\def\mapleft#1{\smash{\mathop{\longleftarrow}\limits^{#1}}}
\def\mapup#1{\Big\uparrow\rlap{$\vcenter {\hbox {$#1$}}$}}
\def\mapdown#1{\Big\downarrow\rlap{$\vcenter {\hbox {$\ssize{#1}$}}$}}
\def\mapne#1{\nearrow\rlap{$\vcenter {\hbox {$#1$}}$}}
\def\mapse#1{\searrow\rlap{$\vcenter {\hbox {$\ssize{#1}$}}$}}
\def\mapr#1{\smash{\mathop{\rightarrow}\limits^{#1}}}
\def\ss{\smallskip}
\def\vp{v_1^{-1}\pi}
\def\at{{\widetilde\alpha}}
\def\sm{\wedge}
\def\la{\langle}
\def\ra{\rangle}
\def\on{\operatorname}
\def\spin{\on{Spin}}
\def\kbar{{\overline k}}
\def\qed{\quad\rule{8pt}{8pt}\bigskip}
\def\ssize{\scriptstyle}
\def\a{\alpha}
\def\bz{{\Bbb Z}}
\def\im{\on{im}}
\def\ct{\widetilde{C}}
\def\ext{\on{Ext}}
\def\sq{\on{Sq}}
\def\eps{\epsilon}
\def\ar#1{\stackrel {#1}{\rightarrow}}
\def\br{{\bold R}}
\def\bC{{\bold C}}
\def\bA{{\bold A}}
\def\bB{{\bold B}}
\def\bD{{\bold D}}
\def\bh{{\bold H}}
\def\bQ{{\bold Q}}
\def\bP{{\bold P}}
\def\bx{{\bold x}}
\def\bo{{\bold{bo}}}
\def\si{\sigma}
\def\Ebar{{\overline E}}
\def\dbar{{\overline d}}
\def\Sum{\sum}
\def\tfrac{\textstyle\frac}
\def\tb{\textstyle\binom}
\def\Si{\Sigma}
\def\w{\wedge}
\def\equ{\begin{equation}}
\def\b{\beta}
\def\G{\Gamma}
\def\g{\gamma}
\def\k{\kappa}
\def\psit{\widetilde{\Psi}}
\def\tht{\widetilde{\Theta}}
\def\psiu{{\underline{\Psi}}}
\def\thu{{\underline{\Theta}}}
\def\aee{A_{\text{ee}}}
\def\aeo{A_{\text{eo}}}
\def\aoo{A_{\text{oo}}}
\def\aoe{A_{\text{oe}}}
\def\fbar{{\overline f}}
\def\endeq{\end{equation}}
\def\sn{S^{2n+1}}
\def\zp{\bold Z_p}
\def\A{{\cal A}}
\def\P{{\cal P}}
\def\cj{{\cal J}}
\def\zt{{\bold Z}_2}
\def\bs{{\bold s}}
\def\bof{{\bold f}}
\def\bq{{\bold Q}}
\def\be{{\bold e}}
\def\Hom{\on{Hom}}
\def\ker{\on{ker}}
\def\coker{\on{coker}}
\def\da{\downarrow}
\def\colim{\operatornamewithlimits{colim}}
\def\zphat{\bz_2^\wedge}
\def\io{\iota}
\def\Om{\Omega}
\def\u{{\cal U}}
\def\e{{\cal E}}
\def\exp{\on{exp}}
\def\wbar{{\overline w}}
\def\xbar{{\overline x}}
\def\ybar{{\overline y}}
\def\zbar{{\overline z}}
\def\ebar{{\overline e}}
\def\nbar{{\overline n}}
\def\rbar{{\overline r}}
\def\et{{\widetilde E}}
\def\ni{\noindent}
\def\coef{\on{coef}}
\def\den{\on{den}}
\def\lcm{\on{l.c.m.}}
\def\vi{v_1^{-1}}
\def\ot{\otimes}
\def\psibar{{\overline\psi}}
\def\mhat{{\hat m}}
\def\exc{\on{exc}}
\def\ms{\medskip}
\def\ehat{{\hat e}}
\def\etao{{\eta_{\text{od}}}}
\def\etae{{\eta_{\text{ev}}}}
\def\dirlim{\operatornamewithlimits{dirlim}}
\def\gt{\widetilde{L}}
\def\lt{\widetilde{\lambda}}
\def\st{\widetilde{s}}
\def\ft{\widetilde{f}}
\def\sgd{\on{sgd}}
\def\lfl{\lfloor}
\def\rfl{\rfloor}
\def\ord{\on{ord}}
\def\gd{{\on{gd}}}
\def\rk{{{\on{rk}}_2}}
\def\nbar{{\overline{n}}}
\def\lg{{\on{lg}}}
\def\N{{\Bbb N}}
\def\Z{{\Bbb Z}}
\def\Q{{\Bbb Q}}
\def\R{{\Bbb R}}
\def\C{{\Bbb C}}
\def\l{\left}
\def\r{\right}
\def\mo{\on{mod}}
\def\vexp{v_1^{-1}\exp}
\def\notimm{\not\subseteq}
\def\Remark{\noindent{\it  Remark}}

\def\*#1{\mathbf{#1}}
\def\0{$\*0$}
\def\1{$\*1$}
\def\22{$(\*2,\*2)$}
\def\33{$(\*3,\*3)$}
\def\ss{\smallskip}
\def\ssum{\sum\limits}
\def\dsum{{\displaystyle{\sum}}}
\def\la{\langle}
\def\ra{\rangle}
\def\on{\operatorname}
\def\o{\on{o}}
\def\U{\on{U}}
\def\lg{\on{lg}}
\def\a{\alpha}
\def\bz{{\Bbb Z}}
\def\eps{\varepsilon}
\def\br{{\bold R}}
\def\bc{{\bold C}}
\def\bN{{\bold N}}
\def\nut{\widetilde{\nu}}
\def\tfrac{\textstyle\frac}
\def\b{\beta}
\def\G{\Gamma}
\def\g{\gamma}
\def\zt{{\bold Z}_2}
\def\zth{{\bold Z}_2^\wedge}
\def\bs{{\bold s}}
\def\bg{{\bold g}}
\def\bof{{\bold f}}
\def\bq{{\bold Q}}
\def\be{{\bold e}}
\def\lline{\rule{.6in}{.6pt}}
\def\xb{{\overline x}}
\def\xbar{{\overline x}}
\def\ybar{{\overline y}}
\def\zbar{{\overline z}}
\def\ebar{{\overline \be}}
\def\nbar{{\overline n}}
\def\rbar{{\overline r}}
\def\Ubar{{\overline U}}
\def\et{{\widetilde e}}
\def\ni{\noindent}
\def\ms{\medskip}
\def\ehat{{\hat e}}
\def\xhat{{\widehat x}}
\def\nbar{{\overline{n}}}
\def\minp{\min\nolimits'}
\def\N{{\Bbb N}}
\def\Z{{\Bbb Z}}
\def\Q{{\Bbb Q}}
\def\R{{\Bbb R}}
\def\C{{\Bbb C}}
\def\el{\ell}
\def\mo{\on{mod}}
\def\dstyle{\displaystyle}
\def\ds{\dstyle}
\def\Remark{\noindent{\it  Remark}}
\title
{2-adic partial Stirling functions and their zeros}
\author{Donald M. Davis}
\address{Department of Mathematics, Lehigh University\\Bethlehem, PA 18015, USA}
\email{dmd1@lehigh.edu}
\date{February 3, 2014}


\maketitle
\begin{abstract} Let $P_n(x)=\frac1{n!}\sum\binom n{2i+1}(2i+1)^x$.
This extends to a continuous function on the 2-adic integers, the $n$th 2-adic
 partial Stirling function. We show that $(-1)^{n+1}P_n$ is the only 2-adically continuous approximation
 to $S(x,n)$, the Stirling number of the second kind.  We present extensive information about the zeros of $P_n$, for which there
 are many interesting patterns. We prove that if $e\ge2$ and
 $2^e+1\le n\le 2^e+4$, then $P_n$ has exactly $2^{e-1}$ zeros, one in each mod $2^{e-1}$
 congruence.
 We study the relationship between the zeros of $P_{2^e+\Delta}$ and $P_\Delta$,
 for $1\le\Delta\le 2^e$, and the convergence of $P_{2^e+\Delta}(x)$
 as $e\to\infty$.
 \end{abstract}

\section{Introduction}\label{intro}

The numbers $$T_n(x):=\sum_{j\text{ odd}}\tbinom njj^x$$
were called partial Stirling numbers in \cite{Lun}, and this terminology (with varying notation) was continued
in \cite{Cl}, \cite{D23}, \cite{partial}, and \cite{You}. Although our results can no doubt be
adapted to odd-primary results, we focus entirely on the prime 2 for simplicity. The 2-exponents $\nu(T_n(x))$
are important in algebraic topology. (\cite{BDSU}, \cite{CK}, \cite{SU2e}, \cite{DS1}, \cite{Lune})
Here and throughout, $\nu(-)$ denotes the exponent of 2 in an integer or rational number or 2-adic integer.

Since the Stirling numbers of the second kind satisfy
$$S(x,n)=\tfrac1{n!}\sum(-1)^{n-j}\tbinom njj^x,\quad x\ge0,$$
it would seem more reasonable to call
\begin{equation}\label{Pdef}P_n(x):={\tfrac1{n!}}\sum_{j\text{ odd}}\tbinom njj^x\end{equation}
the partial Stirling numbers, defined for any integer $x$. Of course, information about either $T_n(x)$ or $P_n(x)$ is easily transformed
into information about the other. We prefer to work with $P_n(x)$ because of its closer relationship with
the Stirling numbers and because of
\begin{prop}\label{P0} For any integer $x$, $\nu(P_n(x))\ge0$ with equality iff $\binom{2x-n-1}{n-1}$ is odd.
\end{prop}
\noindent This implies, of course, that $\nu(T_n(x))\ge\nu(n!)$, which is fine, but less elegant. Proposition \ref{P0} follows easily from the known similar result for $S(x,n)$
when $x\ge n$, that $\nu((-1)^{n+1}P_n(x)-S(x,n))\ge x-\nu(n!)$, and periodicity of $P_n$ given in the next proposition, which we will prove in Section \ref{pfsec}.

\begin{prop}\label{per} Let $\lg(n)=[\log_2(n)]$. For all integers $x$, $$P_n(x+2^t)\equiv P_n(x)\mod 2^{t+1-\lg(n)}.$$\end{prop}

An immediate consequence  is
\begin{cor}  $P_n$ extends to a continuous function $\zt\to\zt$, where $\zt$ denotes
the 2-adic integers, with the usual 2-adic metric $d(x,y)=1/2^{\nu(x-y)}$.\end{cor}
\ni This was pointed out by Clarke in \cite{Cl}, where he also noted that the function $P_n$ is  analytic  on $2\zt+\eps$,
$\eps\in\{0,1\}$. We call $P_n$ a {\it partial Stirling function}.

In \cite{D2}, the author proved that there exist 2-adic integers $x_0$ and $x_1$ such that $\nu(P_5(2x))=\nu(x-x_0)$ and $\nu(P_5(2x+1))=\nu(x-x_1)$
for all $x\in\zt$, and in \cite{Cl}, Clarke noted that $2x_0$ and $2x_1+1$ should be thought of as  2-adic zeros of the
function $P_5$, and these are the only two zeros of   $P_5$ on $\zt$. Recently, in \cite{partial}, the author showed that this sort of behavior occurs frequently
for the functions $P_n$ restricted to certain congruence classes. In this paper, we will continue this investigation of the zeros of $P_n$.
Related to this, we will also discuss $\ds\lim_{e\to\infty}P_{2^e+\Delta}(x)$ for fixed $\Delta>0$.

Next we compare with similar notions for the actual Stirling numbers of the second kind. There are results (\cite{Car}, \cite{Kw}) somewhat similar to
our Proposition \ref{per} saying
$$S(x+2^t,n)\equiv S(x,n)\mod2^{\min(t+1-\lg(n),x-\nu(n!))}$$
if $x\ge n$.
Since, if $n<\!<x<\!<t$, \begin{eqnarray*}\nu(S(x+2^t,n)-S(x,n))&=&
\nu\bigl({\tfrac1{n!}}\sum(-1)^{n-j}\tbinom njj^x(j^{2^t}-1)\bigr)\\
&=&\nu(\tfrac1{n!}\tbinom n22^x)=x-1-\nu((n-2)!),\end{eqnarray*}
we conclude that $x\mapsto S(x,n)$ is not continuous in the 2-adic metric on any domain containing arbitrarily large $x$.
Our partial Stirling function $(-1)^{n+1}P_n$ is the  only 2-adically continuous approximation to $S(-,n)$, which is made precise in the following result.
\begin{prop} For all $x\ge n\ge1$, $(-1)^{n+1}P_n(x)\equiv S(x,n)\mod 2^{x-\nu(n!)}$. Moreover $(-1)^{n+1}P_n$ is
 the only continuous function $f:\zt\to\zt$ for which there exists an
integer $c$ satisfying that for all $x\ge n$, $f(x)\equiv S(x,n)\mod 2^{x-c}$.\end{prop}
\begin{proof} The first part is true since $(-1)^{n+1}P_n(x)-S(x,n)=(-1)^{n+1}\sum\binom n{2j}(2j)^x/n!$. For the second part, we have, for any positive integers $x$ and $L$ with $L$ sufficiently large,
\begin{eqnarray*}&&d(f(x+2^L),(-1)^{n+1}P_n(x+2^L))\\
&\le&d(f(x+2^L),S(x+2^L,n))+d(S(x+2^L,n),(-1)^{n+1}P_n(x+2^L))\\
&\le&1/2^{x+2^L-c}+1/2^{x+2^L-\nu(n!)},\end{eqnarray*}
which approaches 0 as $L\to\infty$. Thus $f(x)=(-1)^{n+1}P_n(x)$ since both functions are continuous. Since positive integers are dense in $\zt$, $f=(-1)^{n+1}P_n$ on $\zt$.\end{proof}

Clarke (\cite{Cl}) conjectured that if, as is often the case, $\nu(P_n(x))=\nu(x-x_0)+c_0$ for some $x_0\in\zt$, $c_0\in\Z$, and all $x$ in a congruence class, then
$\nu(S(x,n))=\nu(x-x_0)+c_0$ on the same congruence class, provided $x\ge n$, and that moreover $\nu(S(x,n))=\nu(P_n(x))$ for all integers $x\ge n$. He pointed out the difficulty
of proving this, which can be thought of as the possibility that $x_0$ might contain  extraordinarily long strings of zeros in its binary expansion.
This will be discussed in  more detail after (\ref{evod8}).

\section{Main theorems}\label{resultssec}
In \cite{partial}, we showed that for $e\ge2$, the functions $P_{2^e+1}$ and $P_{2^e+2}$ have exactly $2^{e-1}$ zeros, one in each mod $2^{e-1}$ congruence class.
One of our main new results is to extend this to $P_{2^e+3}$ and $P_{2^e+4}$. We will prove the following result in Section \ref{pfsec}.
\begin{thm} \label{four} Let $1\le\Delta\le4$, $e\ge2$,  $0\le p<2^{e-1}$, and $p_2$ the mod-2 reduction of $p$.
There exists $x_{e,\Delta,p}\in\zt$ such that for all integers $x$
\begin{equation}\nu(P_{2^e+\Delta}(2^{e-1}x+p))=\nu(x-x_{e,\Delta,p})+\begin{cases}2&\text{if $(\Delta,p_2)=(3,0)$ or $(4,1)$, $e=2$}\\
1&\text{if $(\Delta,p_2)=(3,0)$ or $(4,1)$, $e>2$}\\
0&\text{otherwise.}\end{cases}.\end{equation}
\end{thm}

\begin{cor} \label{2e-1}If $1\le\Delta\le4$ and $e\ge2$, the function $P_{2^e+\Delta}$ has exactly $2^{e-1}$ zeros on $\zt$, given by the 2-adic integers $2^{e-1}x_{e,\Delta,p}+p$
for $0\le p<2^{e-1}$.\end{cor}

It is easy to see, as noted in \cite{D2}, that $P_n$ has no zeros if $1\le n\le4$. Corollary \ref{2e-1} says that $P_n$ has 2 (resp.~4) zeros for $5\le n\le8$ (resp.~$9\le n\le12$).
In Section \ref{zerossec}, we discuss patterns in the zeros of $P_n$, extending work in \cite{D2}. We have located all the zeros of $P_n$ for $n\le101$, and present the results
for $n\le64$ in Tables \ref{t2} and \ref{t3}. The number of zeros of $P_n$ appears to equal, with several exceptions,
\begin{equation}2\biggl[\frac{n-1}4\biggr]+\begin{cases}-2&n\equiv13\ (16)\\ 0&\text{otherwise}.\end{cases}\label{numzeros}\end{equation}
For $n\le 101$, the exceptions are that the number of zeros of $P_n$ is 2 less than that given in (\ref{numzeros}) if $n=21$, $71$, or $90$.
This is a tantalizing aspect of this study---patterns appear, leading perhaps to conjectures, but then there are exceptions. The most striking example of this is
that we were conjecturing that if $1\le\Delta\le2^e$, then $P_{2^e+\Delta}$ has exactly one zero in every mod $2^{e-1}$ congruence class that does not contain a
zero of $P_\Delta$. This fails only once for $2^e+\Delta\le101$: for $x\equiv4$ mod 16, $P_{53}$ has three zeros, while $P_{21}$ has none.
The zeros of $P_{2^e+\Delta}$ in mod $2^{e-1}$ congruence classes in which $P_\Delta$ has zeros are somewhat more complicated, although usually $P_{2^e+\Delta}$ has two zeros in such mod $2^{e-1}$ classes. We will discuss this in Section \ref{zerossec}.

Next we describe another approach related to the zeros of $P_{2^e+\Delta}$. We begin with a simple lemma, which was proved in \cite{Monthly}. Let $\U(n)=n/2^{\nu(n)}$ denote the odd part of $n$.
\begin{lem}\label{oddlem} For all $e\ge1$,  $\U(2^{e-1}!)\equiv\U(2^e!)\mod 2^e$.\end{lem}
\ni Thus there is a well-defined element $\U(2^\infty!):=\lim U(2^e!)$ in $\zt$. Its backwards binary expansion begins $1101000101101\cdots$.

The following theorem will be proved in Section \ref{limsec}.
\begin{thm}\label{Delthm} For $x\ge0$ and $0\le\Delta<2^e$,
\begin{equation}\label{Peq}P_{2^e+\Delta}(x)\equiv{\tfrac1{\U(2^e!)}\tfrac1{\Delta!}}\sum_{j=0}^{\Delta}\tbinom{\Delta}jj^x\mod 2^{e-\max(\lg(x-\Delta)+1,\lg(\Delta)-1)}.\end{equation}
\end{thm}
\noindent Here we ignore $\lg(x-\Delta)$ if $x-\Delta\le0$ (or call it $-\infty$).

This has as an immediate corollary that the 2-adic limit of (\ref{Peq}), as $e\to\infty$, equals the RHS of the following:
\begin{equation}\label{infeq}P_{2^\infty+\Delta}(x):=\lim_{e\to\infty}P_{2^e+\Delta}(x)={\tfrac1{\U(2^\infty!)}\tfrac1{\Delta!}}\sum_{j=0}^{\Delta}\tbinom{\Delta}jj^x,\quad x\ge0.\end{equation}
The novelty here is that we have defined  a function, at least for positive integers $x$, of the form
$P_{2^\infty+\Delta}(x)$
and then related it to the finite sum $\ds\sum_{j=0}^\Delta\tbinom\Delta jj^x$.

We now explain the relevance of (\ref{infeq}) to the zeros of $P_n$.
Note that the RHS of (\ref{infeq}) is a sum over all $j$, not just odd $j$. Since $S(x,n)=0$ when $x<n$ (and $S(x,n)$ is the difference of the sum over
odd $j$ and sum over even $j$), we have
\begin{equation}{\tfrac1{n!}}\sum_{j=0}^n\tbinom njj^x={\tfrac2{n!}}\sum_{j\text{ odd}}\tbinom njj^x\quad \text{if }0\le x<n.\label{2t}\end{equation}
On the other hand, if $x\ge n$, then
\begin{equation}\label{evod}{\tfrac1{n!}}\sum_{j\text{ odd}}{\tbinom njj^x\equiv \tfrac1{n!}}\sum_{j=0}^n\tbinom njj^x\mod 2^{x-\nu(n!)}.\end{equation}
It is likely that, as a consequence of (\ref{evod}), we have
\begin{equation}\label{evod8}\nu({\tfrac1{\Delta!}}\sum_{j=0}^\Delta\tbinom\Delta jj^x)=\nu({\tfrac1{\Delta!}}\sum_{j\text{ odd}}\tbinom\Delta jj^x)=\nu(P_\Delta(x))=\nu(x-x_0)+c_0\quad\text{if }x>\!>\Delta\end{equation}
for $x$ in a congruence class for which the last equality holds for some $x_0\in\zt$.
That it is only ``likely" is due to the possibility that  it might conceivably happen that the zero $x_0$ of $P_\Delta$  satisfies that \begin{equation}\label{name}\nu(x_0-A)\ge A\end{equation} for some large integer $A$. This refers to a long string of
zeros in the binary expansion of $x_0$ mentioned at the end of Section \ref{intro}. Then the inequality
$$\nu({\tfrac1{\Delta!}}\sum_{j=0}^\Delta\tbinom\Delta jj^A-P_\Delta(A))\ge A-\nu(\Delta!)$$
implied by (\ref{evod})
 would not be sufficient to deduce from $\nu(P_\Delta(A))=\nu(A-x_0)+c_0$ that $\nu({\tfrac1{\Delta!}}\sum_{j=0}^\Delta\tbinom\Delta jj^A)=\nu(A-x_0)+c_0$, as desired.
The situation (\ref{name}) would have to happen infinitely often in $x_0$ to create a real problem.

Assuming (\ref{evod8}), it would follow from  (\ref{infeq})
 that the zeros of $P_{2^\infty+\Delta}$ are exactly those of $P_\Delta$.
 Unfortunately, this does not give information about the zeros of $P_{2^e+\Delta}$, since the convergence in (\ref{infeq}) is not uniform.
 Nevertheless, it is interesting that for all positive integers $x$, the sequence $P_{2^e+\Delta}(x)$ converges in $\zt$ as $e\to\infty$.
 This leads one to wonder whether the same thing is true if $x$ is in $\zt-\Z^+$. Quite possibly, the answer is  that a variant of this is true iff $x$ is rational.

Our investigation of this has been focused primarily on the case $\Delta=1$, but we anticipate similar results for any $\Delta>0$.
Our main conjecture here is as follows. Throughout the following, if $x\in\zt$, we let $x_i$ denote the $2^i$-bit of $x$; i.e. $x=\sum_{i\ge0}x_i2^i$
with $x_i\in\{0,1\}$.
\begin{conj}\label{dconj} If, for some $d\ge2$ and $i_0\ge 0$,  $x\in\zt$ satisfies $x_{i+d}=x_i$ for all $i\ge i_0$, then for any $e$, $\ds\lim_{j\to\infty}P_{2^{
e+dj}+1}(x)$ exists in $\zt$.\end{conj}
\ni That is, if $x$ is a 2-adic integer with eventual period $d$ in its binary expansion, then the sequence of $P_{2^e+1}(x)$ as $e\to\infty$ splits into $d$ convergent subsequences.

Table \ref{t1} illustrates this phenomenon. Here we  deal with $z_n:=3\cdot\sum_{i=0}^n2^{3i}$ and tabulate the backwards binary expansion
of $P_{2^e+1}(z_n)$ for $n\ge n_0$, as listed. We list $\nu(P_{2^e+1}(z_n)-P_{2^{e-3}+1}(z_n))$ for emphasis, although these values are clear from comparison
of the 12-bit expansions.

\vfill\eject

\begin{diag}{$P_{2^e+1}(z_n)$ for $n\ge n_0$}\label{t1}
\begin{center}
\begin{tabular}{c|lcc}
$e$&$P_{2^e+1}(z_n)$&$n_0$&$\nu(P_{2^e+1}(z_n)-P_{2^{e-3}+1}(z_n))$\\
\hline
$4$&$011011101000\cdots$&$4$&\\
$5$&$001110101111\cdots$&$4$&\\
$6$&$111101001101\cdots$&$5$&\\
$7$&$011001110011\cdots$&$5$&$4$\\
$8$&$001110110100\cdots$&$6$&$7$\\
$9$&$111101100101\cdots$&$6$&$6$\\
$10$&$011001111001\cdots$&$6$&$8$\\
$11$&$001110110110\cdots$&$7$&$10$\\
$12$&$111101100011\cdots$&$7$&$9$\\
$13$&$011001111000\cdots$&$7$&$11$\\
$14$&$001110110110\cdots$&$8$&$12$\\
$15$&$111101100011\cdots$&$8$&$12$
\end{tabular}
\end{center}
\end{diag}

We now state a more detailed conjecture which implies Conjecture \ref{dconj}.
\begin{conj}\label{specconj} Suppose $x$ is a finite element of $\zt$, and $i_0$ and $d$ are positive integers such
that $x_{i_0}=0$ and $x_{i+d}=x_i$ for all $i\ge i_0$, provided $2^{i+d}\le x$. Denote by $R(x):=\lg(x)+1-(i_0+d)$ the number of repeating bits of $x$.
Then
$$\nu(P_{2^{e+d}+1}(x+1)-P_{2^e+1}(x+1))\ge e-i_0.$$
provided $R(x)\ge2(e-i_0)-1$.
\end{conj}
\begin{proof}[Proof that Conjecture \ref{specconj} implies Conjecture \ref{dconj}.] Let $x$ be as in Conjecture \ref{dconj}. Let $x[n]:=\sum_{i\le n}x_i2^i$, and let $Q_e:=P_{2^e+1}$. We have
\begin{eqnarray*}&&\nu(Q_{e+d}(x)-Q_e(x))\\
&\ge&\min\bigl(\nu(Q_{e+d}(x)-Q_{e+d}(x[n])),\nu(Q_{e+d}(x[n])-Q_e(x[n])),\nu(Q_e(x[n])-Q_e(x))\bigr)\\
&\ge&\min(n+2-e-d,e-i_0,n+2-e),\end{eqnarray*}
provided $n-d\ge 2e-i_0-2$,
using Proposition \ref{per} for the first and last parts. For any $e$, we can make this $\ge e-i_0$ by choosing $n$ sufficiently large.
Thus the sequence  $\la Q_{e+jd}(x)\ra$ is Cauchy.\end{proof}

Conjecture \ref{specconj} has been verified for $i_0=5$, $2\le d\le 7$, $6\le e\le 9$, and many values of $x$ mod $2^{i_0}$.

\section{Zeros of $P_n$}\label{zerossec}
In this section, we describe various facts about the zeros of the functions $P_n$. Most of these  can be considered to be extensions
of results of \cite{D2}, but the emphasis here is on the zeros rather than divisibility.

We begin with a broad outline of our proofs,
but defer most details to the following section.
This outline is needed to understand certain aspects of our tabulated results.

One of our main tools is the following result, which is a slight refinement of \cite[Theorem 1]{D2}.
Here we use the notation that $\min'(a,b)=\min(a,b)$ if $a\ne b$, while $\min'(a,a)>a$.
\begin{lem} $($\cite{D2}$)$ A function $f:\Z\to\Z\cup\{\infty\}$ satisfies that there exists $z_0\in\zt$ such that
$f(x)=\nu(x-z_0)$ for all integers $x$ iff $f(0)\ge0$ and for all $x\in\Z$ and all $d\ge0$,
$$f(x+2^d)={\min}'(f(x),d).$$\label{D2lem}\end{lem}

The difference between this and the result of \cite{D2} is that here we do not assume at the outset that $f(x)\ge0$ for all $x$. As can be seen from the
proof in \cite{D2}, all that is required is $f(0)\ge0$ since $z_0=2^{e_0}+2^{e_1}+\cdots$ with $e_0<e_1<\cdots$ and $e_0=f(0)$, $e_i=f(2^{e_0}+\cdots+2^{e_{i-1}})$.

\begin{cor} \label{D2cor} If $g:\Z\to\Q\cup\{\infty\}$ satisfies that there exists an integer $c$ such that $\nu(g(0))\ge c$ and, for all integers $x$ and $d$ with $d\ge0$,
$$\nu(g(x+2^d)-g(x))=d+c,$$
then there exists $z_0\in\zt$ such that, for all $x\in\Z$,
$$\nu(g(x))=\nu(x-z_0)+c.$$
\end{cor}
\begin{proof} The hypothesis implies that
$$\textstyle\nu(g(x+2^d))=\min'(\nu(g(x)),d+c).$$
Apply the lemma to $f(x)=\nu(g(x))-c$.\end{proof}

Let
\begin{equation}\label{phidef}\Phi_n(s)={\tfrac1{n!}}\sum_i\tbinom n{2i+1}(2i)^s.\end{equation}
Since
\begin{eqnarray*}&&P_n(2^{e-1}(x+2^d)+p)-P_n(2^{e-1}x+p)\\
 &=&{\tfrac1{n!}}\sum_i\tbinom n{2i+1}(2i+1)^{2^{e-1}x+p}((2i+1)^{2^{e-1+d}}-1)\\
 &=&\sum_{k\ge0}{\tbinom{2^{e-1}x+p}k}\sum_{j>0}\tbinom{2^{e-1+d}}j\Phi_n(j+k),\end{eqnarray*}
 Corollary \ref{D2cor} implies that to show
 $$\nu(P_n(2^{e-1}x+p))=\nu(x-x_0)+c$$
 for some $x_0\in\zt$, it suffices to prove that
 \begin{equation}\label{newjkeq}\nu\bigl(\sum_{k\ge0}{\tbinom{2^{e-1}x+p}k}\sum_{j>0}{\tfrac1{2^d}\tbinom{2^{e-1+d}}j}\Phi_n(j+k)\bigr)=c\end{equation}
 for all integers $x$ and $d$ with $d\ge0$ (and that $\nu(P_n(p))\ge c$). The study of (\ref{newjkeq}) will occupy much of our effort.

 Table \ref{t2} describes
the location of the zeros of $P_n$ for $17\le n\le 32$. This information was given, in a different form, in \cite[Table 1.3,1.4]{D2}.

\begin{center}
\begin{minipage}{6.5in}
\begin{diag}\label{t2}{Zeros of $P_n$ in $(p$ mod $8)$, $17\le n\le32$}
\begin{center}
\begin{picture}(200,350)
\def\mp{\multiput}
\def\elt{\circle*{3}}
\def\cir{\circle{3}}
\mp(30,0)(20,0){9}{\line(0,1){320}}
\mp(30,0)(0,20){17}{\line(1,0){160}}
\mp(40,310)(20,0){8}{\cir}
\mp(40,290)(20,0){8}{\cir}
\mp(40,270)(20,0){8}{\cir}
\mp(80,270)(80,0){2}{\elt}
\mp(40,250)(40,0){4}{\cir}
\mp(60,250)(40,0){4}{\elt}
\mp(40,230)(20,0){3}{\cir}
\put(100,230){\elt}
\mp(140,230)(20,0){2}{\cir}
\mp(175,230)(10,0){2}{\elt}
\mp(40,210)(20,0){4}{\cir}
\mp(115,210)(10,0){2}{\elt}
\mp(135,205)(0,10){2}{\elt}
\mp(160,210)(20,0){2}{\cir}
\mp(40,190)(20,0){5}{\cir}
\mp(80,190)(40,0){2}{\elt}
\mp(135,185)(0,10){2}{\elt}
\mp(155,185)(0,10){2}{\elt}
\put(180,190){\cir}
\mp(40,170)(20,0){6}{\cir}
\mp(60,170)(20,0){3}{\elt}
\put(140,170){\elt}
\mp(155,165)(0,10){2}{\elt}
\mp(175,165)(0,10){2}{\elt}
\mp(35,150)(10,0){2}{\cir}
\mp(60,150)(20,0){4}{\cir}
\put(100,150){\elt}
\mp(135,150)(20,0){3}{\elt}
\mp(145,150)(20,0){3}{\elt}
\mp(35,130)(10,0){4}{\cir}
\mp(80,130)(20,0){4}{\cir}
\mp(155,125)(0,10){2}{\elt}
\mp(175,130)(10,0){2}{\elt}
\mp(35,110)(10,0){4}{\cir}
\mp(75,110)(10,0){2}{\elt}
\mp(100,110)(20,0){3}{\cir}
\put(160,110){\elt}
\mp(175,110)(10,0){2}{\elt}
\mp(35,90)(10,0){8}{\cir}
\mp(55,90)(10,0){2}{\elt}
\mp(95,90)(10,0){2}{\elt}
\mp(120,90)(20,0){4}{\cir}
\mp(140,90)(40,0){2}{\elt}
\mp(35,70)(10,0){6}{\cir}
\put(105,70){\elt}
\mp(115,70)(10,0){2}{\elt}
\mp(140,70)(20,0){2}{\cir}
\put(180,70){\elt}
\mp(35,50)(10,0){12}{\cir}
\mp(115,50)(10,0){4}{\elt}
\mp(160,50)(20,0){2}{\cir}
\put(40,30){\elt}
\mp(55,30)(10,0){8}{\cir}
\mp(135,30)(10,0){4}{\elt}
\mp(75,30)(40,0){3}{\elt}
\mp(85,30)(40,0){3}{\elt}
\put(180,30){\cir}
\mp(40,10)(20,0){2}{\cir}
\put(60,10){\elt}
\mp(75,10)(10,0){12}{\elt}
\put(118,335){$p$}
\put(-8,187){$n$}
\put(38,322){$0$}
\put(58,322){$1$}
\put(78,322){$2$}
\put(98,322){$3$}
\put(118,322){$4$}
\put(138,322){$5$}
\put(158,322){$6$}
\put(178,322){$7$}
\put(5,307){$17$}
\put(5,287){$18$}
\put(5,267){$19$}
\put(5,247){$20$}
\put(5,227){$21$}
\put(5,207){$22$}
\put(5,187){$23$}
\put(5,167){$24$}
\put(5,147){$25$}
\put(5,127){$26$}
\put(5,107){$27$}
\put(5,87){$28$}
\put(5,67){$29$}
\put(5,47){$30$}
\put(5,27){$31$}
\put(5,7){$32$}
\end{picture}
\end{center}
\end{diag}
\end{minipage}
\end{center}

We now explain the table. We temporarily refer to either a $\bullet$ or a $\circ$ as a ``dot.''
The dots in $(n,p)$ represent the zeros of $P_n(z)$ for which $z\equiv p$ mod 8. A dot in the center of  square $(n,p)$ means that
$P_n$ has a zero of the form $8x_0+p$ for some $x_0\in\zt$, and that, moreover, there is an integer $c$ such that
\begin{equation}\nu(P_n(8x+p))=\nu(x-x_0)+c\label{kap}\end{equation}
for all integers $x$.
 Two horizontally-displaced dots
in a box mean that $P_n$ has zeros of the form $16x_0+p$ and $16x_1+8+p$, and analogues of (\ref{kap}) hold for
$\nu(P_n(16x+p))$ and $\nu(P_n(16x+8+p))$. Two vertically-displaced dots on the left side of a box mean that $P_n$ has zeros of the form
$32x_0+p$ and $32x_1+16+p$, with analogues of (\ref{kap}). The single dot on the right side of $(29,3)$ is a zero of the form $16x_0+11$.

Next we explain the difference between $\circ$ and $\bullet$ in the table.  In order to prove (\ref{kap}), we would like to prove
(\ref{newjkeq}) with $e=4$. The cases indicated by a single $\circ$ are those in which, for all $k\ge0$ and $j>0$
$$\nu\bigl({\tbinom{8x+p}k\tfrac1{2^d}\tbinom{2^{d+3}}j}\Phi_n(j+k)\bigr)\ge c$$
with equality for a unique pair $(k,j)$. Cases with two horizontally-displaced $\circ$'s are analogous with $8x+p$ replaced by $16x+p$ and $16x+8+p$, except that here
the minimum value will occur for a unique $(j,k)$ for $16x+p$, and for three $(j,k)$'s for $16x+8+p$.
The $\bullet$'s in the table are zeros of $P_n$ in which some of the terms $\binom{8x+p}k\frac1{2^d}\binom{2^{d+3}}j\Phi_n(j+k)$ have 2-exponent
smaller than that of their sum, and so more complicated combinations, involving odd factors of some terms, must be considered.

Next we present the analogue of Table \ref{t2} for $33\le n\le 64$. The main reason for including such a large table is
to illustrate the great deal of regularity, marred by a few exceptions. After presenting the table, we will explain the aspects in which
it differs from Table \ref{t2}.

\begin{center}
\begin{minipage}{6.5in}
\begin{diag}\label{t3}{Zeros of $P_n$ in $(p$ mod $16)$, $33\le n\le64$}
\begin{center}
\begin{picture}(400,660)
\def\mp{\multiput}
\def\elt{\circle*{3}}
\def\cir{\circle{3}}
\mp(20,0)(20,0){17}{\line(0,1){640}}
\mp(20,0)(0,20){33}{\line(1,0){320}}
\mp(30,630)(20,0){16}{\cir}
\mp(30,610)(20,0){16}{\cir}
\mp(30,590)(20,0){16}{\cir}
\mp(70,590)(80,0){4}{\elt}
\mp(30,570)(40,0){8}{\cir}
\mp(50,570)(40,0){8}{\elt}
\mp(30,550)(20,0){12}{\cir}
\mp(90,550)(20,0){2}{\elt}
\mp(170,550)(80,0){2}{\elt}
\mp(265,550)(10,0){2}{\elt}
\mp(290,550)(20,0){2}{\cir}
\mp(325,550)(10,0){2}{\elt}
\mp(30,530)(20,0){4}{\cir}
\mp(105,530)(10,0){2}{\elt}
\mp(130,530)(140,0){2}{\elt}
\mp(150,530)(20,0){6}{\cir}
\mp(295,525)(0,10){2}{\elt}
\mp(310,530)(20,0){2}{\cir}
\mp(30,510)(20,0){13}{\cir}
\mp(70,510)(160,0){2}{\elt}
\mp(110,510)(20,0){3}{\elt}
\put(270,510){\elt}
\mp(295,505)(0,10){2}{\elt}
\mp(315,505)(0,10){2}{\elt}
\put(330,510){\cir}
\mp(30,490)(20,0){14}{\cir}
\mp(50,490)(20,0){3}{\elt}
\mp(130,490)(20,0){3}{\elt}
\mp(210,490)(20,0){3}{\elt}
\put(290,490){\elt}
\mp(315,485)(0,10){2}{\elt}
\mp(335,485)(0,10){2}{\elt}
\mp(30,470)(20,0){5}{\cir}
\mp(90,470)(160,0){2}{\elt}
\mp(125,470)(10,0){2}{\elt}
\mp(165,470)(10,0){4}{\elt}
\mp(210,470)(20,0){5}{\cir}
\mp(290,470)(20,0){3}{\elt}
\put(142,466){$2$}
\mp(30,450)(20,0){6}{\cir}
\mp(150,450)(180,0){2}{\elt}
\mp(165,450)(10,0){6}{\elt}
\mp(230,450)(20,0){4}{\cir}
\mp(315,445)(0,10){2}{\elt}
\mp(30,430)(20,0){7}{\cir}
\mp(70,430)(80,0){2}{\elt}
\mp(165,430)(10,0){8}{\elt}
\mp(250,430)(20,0){3}{\cir}
\mp(310,430)(20,0){2}{\elt}
\mp(30,410)(40,0){4}{\cir}
\mp(50,410)(40,0){4}{\elt}
\mp(185,410)(10,0){8}{\elt}
\mp(270,410)(40,0){2}{\cir}
\mp(290,410)(40,0){2}{\elt}
\mp(30,390)(20,0){7}{\cir}
\mp(90,390)(20,0){2}{\elt}
\mp(170,390)(80,0){3}{\elt}
\put(182,386){$2$}
\mp(205,390)(10,0){4}{\elt}
\mp(290,390)(20,0){2}{\cir}
\mp(265,390)(10,0){2}{\elt}
\mp(30,370)(20,0){8}{\cir}
\mp(110,370)(20,0){2}{\elt}
\mp(185,370)(10,0){2}{\elt}
\mp(205,365)(0,10){2}{\elt}
\mp(225,370)(10,0){8}{\elt}
\mp(310,370)(20,0){2}{\cir}
\mp(30,350)(20,0){9}{\cir}
\mp(70,350)(40,0){4}{\elt}
\put(130,350){\elt}
\mp(205,345)(0,10){2}{\elt}
\mp(225,345)(0,10){2}{\elt}
\mp(245,350)(10,0){8}{\elt}
\put(330,350){\cir}
\mp(30,330)(20,0){10}{\cir}
\mp(50,330)(20,0){7}{\elt}
\put(210,330){\elt}
\mp(225,325)(0,10){2}{\elt}
\mp(245,325)(0,10){2}{\elt}
\mp(265,330)(10,0){8}{\elt}
\mp(25,310)(10,0){2}{\cir}
\mp(50,310)(20,0){8}{\cir}
\mp(90,310)(40,0){3}{\elt}
\put(150,310){\elt}
\mp(205,310)(10,0){14}{\elt}
\mp(25,290)(10,0){4}{\cir}
\mp(70,290)(20,0){8}{\cir}
\mp(150,290)(20,0){2}{\elt}
\mp(225,285)(0,10){2}{\elt}
\mp(245,290)(10,0){10}{\elt}
\mp(25,270)(10,0){6}{\cir}
\mp(65,270)(10,0){2}{\elt}
\mp(90,270)(20,0){8}{\cir}
\mp(150,270)(20,0){2}{\elt}
\put(230,270){\elt}
\mp(245,270)(10,0){10}{\elt}
\mp(25,250)(10,0){8}{\cir}
\mp(45,250)(10,0){2}{\elt}
\mp(85,250)(10,0){2}{\elt}
\mp(110,250)(40,0){4}{\cir}
\mp(130,250)(40,0){4}{\elt}
\mp(265,250)(10,0){8}{\elt}
\mp(25,230)(10,0){6}{\cir}
\mp(85,230)(10,0){2}{\elt}
\mp(105,225)(0,10){2}{\elt}
\put(115,230){\elt}
\mp(130,230)(20,0){2}{\cir}
\put(162,226){$2$}
\mp(190,230)(20,0){3}{\cir}
\mp(250,230)(20,0){2}{\elt}
\mp(285,230)(10,0){6}{\elt}
\mp(25,210)(10,0){8}{\cir}
\mp(105,210)(10,0){3}{\elt}
\mp(135,205)(0,10){2}{\elt}
\mp(150,210)(20,0){6}{\cir}
\mp(265,210)(10,0){2}{\elt}
\put(290,210){\elt}
\mp(305,210)(10,0){4}{\elt}
\mp(25,190)(10,0){10}{\cir}
\mp(65,190)(10,0){2}{\elt}
\mp(105,190)(10,0){2}{\elt}
\mp(125,190)(20,0){2}{\elt}
\mp(135,185)(0,10){2}{\elt}
\mp(155,185)(0,10){2}{\elt}
\mp(170,190)(20,0){8}{\cir}
\mp(230,190)(40,0){3}{\elt}
\put(290,190){\elt}
\mp(325,190)(10,0){2}{\elt}
\mp(25,170)(10,0){12}{\cir}
\mp(45,170)(10,0){6}{\elt}
\mp(125,170)(10,0){2}{\elt}
\mp(145,170)(20,0){2}{\elt}
\mp(155,165)(0,10){2}{\elt}
\mp(175,165)(0,10){2}{\elt}
\mp(190,170)(20,0){8}{\cir}
\mp(210,170)(20,0){3}{\elt}
\mp(290,170)(20,0){3}{\elt}
\mp(25,150)(10,0){18}{\cir}
\mp(125,150)(10,0){8}{\elt}
\mp(85,150)(10,0){2}{\elt}
\mp(210,150)(20,0){4}{\cir}
\put(250,150){\elt}
\mp(285,150)(10,0){2}{\elt}
\mp(325,150)(10,0){2}{\elt}
\put(302,146){$2$}
\mp(25,130)(10,0){12}{\cir}
\put(145,130){\elt}
\mp(155,125)(0,10){2}{\elt}
\mp(165,130)(10,0){6}{\elt}
\mp(185,130)(10,0){4}{\cir}
\mp(230,130)(20,0){4}{\cir}
\put(310,130){\elt}
\mp(325,130)(10,0){2}{\elt}
\mp(25,110)(10,0){22}{\cir}
\mp(145,110)(10,0){8}{\elt}
\mp(65,110)(10,0){2}{\elt}
\mp(250,110)(20,0){3}{\cir}
\put(310,110){\elt}
\mp(325,110)(10,0){2}{\elt}
\mp(25,90)(40,0){6}{\cir}
\mp(35,90)(40,0){6}{\cir}
\mp(45,90)(40,0){6}{\elt}
\mp(165,90)(10,0){8}{\elt}
\mp(55,90)(40,0){6}{\elt}
\mp(270,90)(40,0){2}{\cir}
\mp(290,90)(40,0){2}{\elt}
\mp(25,70)(10,0){6}{\cir}
\mp(95,70)(10,0){3}{\elt}
\mp(105,70)(10,0){18}{\cir}
\mp(165,70)(10,0){12}{\elt}
\mp(245,70)(10,0){2}{\elt}
\mp(290,70)(20,0){2}{\cir}
\put(330,70){\elt}
\mp(25,50)(10,0){28}{\cir}
\mp(105,50)(10,0){4}{\elt}
\mp(185,50)(10,0){10}{\elt}
\mp(310,50)(20,0){2}{\cir}
\mp(30,30)(300,0){2}{\cir}
\put(30,30){\elt}
\mp(45,30)(40,0){7}{\cir}
\mp(55,30)(40,0){7}{\cir}
\mp(65,30)(40,0){7}{\elt}
\mp(125,30)(10,0){2}{\elt}
\mp(185,30)(10,0){12}{\elt}
\mp(75,30)(40,0){7}{\elt}
\mp(30,10)(20,0){2}{\cir}
\put(50,10){\elt}
\mp(65,10)(10,0){28}{\elt}

\put(28,642){$0$}
\put(48,642){$1$}
\put(68,642){$2$}
\put(88,642){$3$}
\put(108,642){$4$}
\put(128,642){$5$}
\put(148,642){$6$}
\put(168,642){$7$}
\put(188,642){$8$}
\put(208,642){$9$}
\put(224,642){$10$}
\put(244,642){$11$}
\put(264,642){$12$}
\put(284,642){$13$}
\put(304,642){$14$}
\put(324,642){$15$}
\put(4,627){$33$}
\put(4,607){$34$}
\put(4,587){$35$}
\put(4,567){$36$}
\put(4,547){$37$}
\put(4,527){$38$}
\put(4,507){$39$}
\put(4,487){$40$}
\put(4,467){$41$}
\put(4,447){$42$}
\put(4,427){$43$}
\put(4,407){$44$}
\put(4,387){$45$}
\put(4,367){$46$}
\put(4,347){$47$}
\put(4,327){$48$}
\put(4,307){$49$}
\put(4,287){$50$}
\put(4,267){$51$}
\put(4,247){$52$}
\put(4,227){$53$}
\put(4,207){$54$}
\put(4,187){$55$}
\put(4,167){$56$}
\put(4,147){$57$}
\put(4,127){$58$}
\put(4,107){$59$}
\put(4,87){$60$}
\put(4,67){$61$}
\put(4,47){$62$}
\put(4,27){$63$}
\put(4,7){$64$}

\end{picture}
\end{center}
\end{diag}
\end{minipage}
\end{center}

A box $(n,p)$ in Table \ref{t3} with one dot on the left side and two vertically-placed dots on the right says that, in $(p$ mod 16), $P_n$ has zeros of the form $32x_0+p$, $64x_1+16+p$, and $64x_2+48+p$ with formulas analogous to (\ref{kap}) in each congruence class. If box $(n,p)$ has a number 2 on its left side, $P_n$ has two zeros in $(p$ mod 16) of the form $2^tx_0+32+p$ and $2^tx_1+2^{t-1}+32+p$ for $t=9$, 7, 9, 8, if $n=41$, 45, 53, 57, resp.

In \cite[Theorem 1.7]{partial}, we proved a general result describing a large family of cases in which, if $e=\lg(n-1)$, $P_n(2^{e-1}x+p)$ has a single zero, due to $\nu\bigl(\binom{2^{e-1}x+p}k\binom{2^{e-1}}j\Phi_n(j+k)\bigr)$ obtaining its minimum value for a unique $(j,k)$. For $17\le n\le64$,
these are the cases where the box $(n,p)$ in Table \ref{t2} or \ref{t3} has a single $\circ$ in the center of the box. We restate the result in a slightly simpler way here.
Here we introduce the notation $\a(n)$ for the number of 1's in the binary expansion of $n$. This notation will occur frequently in our proofs, mainly due to the well-known formulas
$$\nu(n!)=n-\a(n)\text{\quad and\quad}\nu\tbinom mn=\a(n)+\a(m-n)-\a(m),$$
which we will use without comment.
\begin{thm}\label{single} $($\cite[1.7]{partial}$)$ Let $e=\lg(n-1)$ and $t=\lg(n-2^e)$. Suppose $\max(0,n-2^e-2^{e-1})\le p<2^{e-1}$ and $\binom{n-1-p}p$ is odd, and let $p_0$ denote the mod $2^t$ reduction of $p$. Suppose
$$q=p+\eps\cdot2^{\nu(n)-1}+b\cdot2^{t+1}$$
for $\eps\in\{0,1\}$ and $b\ge0$, with $q<2^{e-1}$. Then
$$\nu\biggl(\binom{2^{e-1}x+q}k\binom{2^{e-1}}j\Phi_n(j,k)\biggr)\ge\a(n)-2-\a(p_0)$$
with equality iff $(j,k)=(2^{e-1},p_0)$.\end{thm}
\begin{cor} If $n$, $p_0$, and $q$ are as above, then there exists $x_0\in\zt$ such that for all integers $x$
$$\nu(P_n(2^{e-1}x+q))=\nu(x-x_0)+\a(n)-2-\a(p_0).$$
Hence, $P_n$ has a unique zero, $2^{e-1}x_0+q$, in $(q$ mod $2^{e-1})$.\label{corsing}
\end{cor}

\begin{rmk} \label{rem}{\rm The boxes in Tables \ref{t2} and \ref{t3} with a single $\circ$ in the center are all the cases in this range in which, if $e=\lg(n-1)$, $P_n(2^{e-1}x+p)$ has a single zero, due to $\nu\bigl(\binom{2^{e-1}x+p}k\binom{2^{e-1}}j\Phi_n(j+k)\bigr)$ obtaining its minimum value for a unique $(j,k)$. All of these cases fit into families that work for all $e$.
However, there are several of these which are not covered by Theorem \ref{single}. When $n=2^e+1$ or $2^e+2$, not all values of $p$ are handled by Theorem \ref{single}, but they are handled by Theorem \ref{four}. Also the case $(n,p)=(2^e+2^e,0)$ is not covered by Theorem \ref{single}, but it is easily proved, using Proposition \ref{qprop1}, that if $j>0$}
$$\nu\bigl(\tbinom{2^{e-1}x}k\tbinom{2^{e-1}}j\Phi_{2^{e+1}}(j+k)\bigr)\ge0\text{ with equality iff }(j,k)=(2^{e-1},0),$$
{\rm and hence $P_{2^{e+1}}$ has a single zero in (0 mod $2^{e-1}$). We conjecture that, for all $e$ and $n$ with $\lg(n-1)=e$,  these families together provide all cases in which $P_n(2^{e-1}x+p)$ has a single zero due to a single $(j,k)$. }\end{rmk}

A similar result describes cases in which (with $e=\lg(n-1)$) a mod $2^{e-1}$ class splits into two mod-$2^e$ classes with each having a single zero
of $P_n$, and the first of the two determined by a unique $(j,k)$, and the second by three $(j,k)$'s. These are represented in Tables \ref{t2} and \ref{t3} by boxes with two
horizontally-displaced $\circ$'s.
\begin{thm}\label{double} Let $3\cdot2^{e-1}< n<2^{e+1}$. Suppose $0\le p\le[(n-3\cdot2^{e-1})/2]$, $(n,p)\ne(2^{e+1}-1,0)$, and $\binom{n-1-p}p$ is odd. Let $\ell=\lg(2^{e+1}-(n-p))$. If $q=p+\eps\cdot2^{\nu(n)-1}$ for $\eps\in\{0,1\}$, then, if $\delta\in\{0,1\}$,  $x\in\Z$, $k\ge0$, and $j>0$,
\begin{equation}\label{alph}\nu\biggl(\binom{2^ex+\delta2^{e-1}+q}k\binom{2^{e}}j\Phi_n(j+k)\biggr)\ge\a(n)-1-\a(p)\end{equation}
with equality iff $(j,k)=(2^e-2^\ell,p)$ or $\delta=1$ and
$$(j,k)\in\{(2^{e-1}-2^\ell,2^{e-1}+p),\ (2^{e-1},2^{e-1}+p-2^\ell )\}.$$
\end{thm}
\begin{cor} Let $n$, $p$, $q$ and $\delta$ be as above. Then there exists a 2-adic integer $x_\delta$ such that for  all integers $x$
$$\nu(P_n(2^ex+2^{e-1}\delta+q))=\nu(x-x_\delta)+\a(n)-1-\a(p).$$
Hence $P_n$ has a unique zero, $2^ex_\delta+2^{e-1}\delta+q$, in $(2^{e-1}\delta +q$ mod $2^e)$.\label{cordbl}
\end{cor}

There are two types of proofs which we require, both involving the verification of (\ref{newjkeq}).
One is to establish the zeros of $P_n$ for a specific $n$ and specific congruence class, as are presented in Tables \ref{t2} and \ref{t3}.
The other is to prove  general results, such as Theorems \ref{four} and \ref{double}, which apply to infinitely many values of $n$.
The first of these types can be accomplished using {\tt Maple}, using Proposition \ref{DSprop} to limit the set of values of $(j,k)$ that need to be checked.
This will be discussed in the remainder of this section. The second type involves using general results about $\nu(\Phi_n(s))$; this will be discussed in Section \ref{pfsec}.

The following useful result is a restatement of \cite[Theorem 3.4]{DS1}.
\begin{prop}\label{DSprop} If $n$ and $s$ are nonnegative integers, then $\nu(\Phi_n(s))\ge s-[n/2]$.
\end{prop}

The results in Tables \ref{t2} and \ref{t3} are discovered by having {\tt Maple} compute the numbers $\nu(P_n(2^{e-1}x+p))$ for many values of $x$.
We illustrate the process of discovery and proof with two examples.

Even though (almost) all the $\circ$'s in the tables are proved in Corollaries \ref{corsing} and \ref{cordbl}, we briefly sketch how one is discovered and proved.
We consider the box $(n,p)=(29,2)$, which contains a double $\circ$.
Indeed, $P_{29}$ turns out to have two zeros in (2 mod 8), one in 2 mod 16, and one in 10 mod 16. The likelihood of this is seen when {\tt Maple}
computes $\nu(P_{29}(16x+2))$ for consecutive values of $x$, and the values are 2, 3, 2, 4, 2, 3, 2, 5, 2,$\ldots$, which is the pattern of
$\nu(x-x_0)+2$, and a similar result is seen for $\nu(P_{29}(16x+10))$. To prove that $\nu(P_{29}(16x+10))=\nu(x-x_0)+2$ for some 2-adic integer $x_0$, it suffices to show, by the argument leading to
(\ref{newjkeq}), that
for all $j>0$ and $k\ge0$
\begin{equation}\label{59ineq}\nu\tbinom{16x+10}k+4-\nu(j)+\nu(\Phi_{29}(j+k))\ge2\end{equation}
with equality occurring for an odd number of pairs $(j,k)$.
We use here and later that $\nu\binom{2^t}j=t-\nu(j)$ for $1\le j\le2^t$.

Since Proposition \ref{DSprop} implies that
$\nu(\Phi_{29}(j+k))\ge j+k-14$,
strict inequality holds in (\ref{59ineq}) provided $j+k\ge17$. Thus we are reduced to a finite number of verifications, and {\tt Maple} shows that (\ref{59ineq}) holds with equality iff
$(j,k)=(4,10)$, $(12,2)$, or $(8,6)$.

Most\footnote{Several boxes in Table \ref{t3} with two horizontally-displaced $\bullet$'s have minimal $\nu(\binom{2^5x+p}k\binom{2^5}j\Phi_n(j+k))$ occurring an odd number of times but more than once; these do not seem to fit into an easily-proved general formula.} of the $\bullet$'s in the tables are zeros of $P_n$ in which some of the $(j,k)$-terms of (\ref{newjkeq}) have 2-exponent
smaller than that of their sum, and so more than just the 2-exponent of the terms must be considered.
We illustrate with the proof for a typical such case, the left dot in $(31,2)$. The sequence of $\nu(P_{31}(16x+2))$ is 7, 8, 7, 9, 7, 8, 7, 10, 7,$\ldots$, and so we wish to prove
\begin{equation}\label{7}\nu\bigl(\sum_{k\ge0}{\tbinom{16x+2}k}\sum_{j>0}{\tfrac1{2^d}\tbinom{2^{d+4}}j}\Phi_{31}(j+k)\bigr)=7\end{equation}
for all $x\in\Z$ and all $d\ge0$.
 {\tt Maple} verifies that
\begin{equation}\label{31ineq}\nu\tbinom{16x+2}k+4-\nu(j)+\nu(\Phi_{31}(j+k))\ge4\end{equation}
for all $j>0$, $k\ge0$ with $j+k\le 19$, which by Proposition \ref{DSprop} are the only values of $j$ and $k$ that we need to consider, and for $0\le x\le7$. There are many pairs $(j,k)$ for which this
value equals 4, 5, 6, and 7, and these combine in a complicated way to give 2-exponent 7 for the sum. {\tt Maple} can easily enough check this value for the sum,
but there are two things that must be considered in giving a proof valid for all integers $x$ and $d$.  For $(j,k)$-summands with $\nu=4$, the mod 16 value of the odd part of $\frac1{2^d}\binom{2^{4+d}}j$ for various values of $d$ must be taken into account, and, similarly, changing $x$ causes changes in  $\binom{16x+2}k$,
which are essential in proving that (\ref{7}) holds for all $x$ and $d$.

Similarly to Lemma \ref{lem1}, one easily proves
\begin{equation}\label{d+1}\nu\biggl(\frac1{2^{d+1}}\binom{2^{d+1+b}}j-\frac1{2^d}\binom{2^{d+b}}j\biggr)=2b+d-\lg(j-1)-\nu(j),\end{equation}
and so the odd factors for $d$ and $d+1$ are congruent mod $2^{b+d-\lg(j-1)}$. As we have $b=4$ and $\lg(j-1)\le 4$, these odd factors
will be congruent mod 16 provided $d\ge4$. Thus the validity of (\ref{7}) for $d\le4$, which is checked by {\tt Maple}, implies its validity for all $d$.
A similar argument shows that changing $x$ by a multiple of 8 changes (\ref{7}) by a multiple of $2^8$, and so verifying (\ref{7}) for $0\le x\le 7$
implies it for all $x$.

A similar analysis proves the blanks in Tables \ref{t2} and \ref{t3}, namely that there are no zeros in certain congruences.
We illustrate with the case $n=23$, right side of column 6. {\tt Maple} verifies that $\nu(P_{23}(16x+14))=4$ for $0\le x\le3$.
We will prove that
\begin{equation}\label{64}\nu(P_{23}(16x+64i+14)-P_{23}(16x+14))\ge5.\end{equation} These together imply that
$\nu(P_{23}(16x+14))=4$ for all integers $x$, and hence $P_{23}$ has no zeros in (14 mod 16).

We write
$$P_{23}(16x+14)=\sum\tbinom{16x+14}k\Phi_{23}(k).$$
Using Proposition \ref{DSprop}, we easily see that terms with $k>14$ have 2-exponent $\ge5$. Using {\tt Maple}, we see that all terms have
2-exponent $\ge2$. Similarly to the proof of \ref{lem1}, we can prove that for $0\le k\le14$
$$\nu\bigl(\tbinom{16x+64i+14}k-\tbinom{16x+14}k\bigr)\ge 3+\nu(i)+\nu\tbinom{16x+14}k,$$
implying (\ref{64}).

Using the ideas discussed in the above examples,
{\tt Maple} can systematically find and prove all the results in Tables \ref{t2} and \ref{t3}.
In each case, {\tt Maple} discovers the value of $c$ by computing a sequence of values of $\nu(P_n(2^{e-1}x+p))$. The following remarkable
formula, obtained by inspecting the $c$-values obtained in all cases,  gives the value of $c$ in every case of Table \ref{t2} and \ref{t3}.
\begin{equation}\label{cgen}\nu(P_n(z))=\sum\nu(z-z_i)-\nu([\tfrac{n-1}2]!)+\begin{cases}\nu([\tfrac{n+1}2])&n\equiv0,3\ (4)\text{ and }n+z\text{ even}\\
\min(15,2\nu(z-148))&n=21\\
\min(9,2\nu(z-19))&n=29\\
\min(8,2\nu(z-11))&n=45\\
\min(10,2\nu(z-3)&n=61\\
0&\text{otherwise,}\end{cases}\end{equation}
where the sum is taken over all the zeros $z_i$ of $P_n$.

We illustrate this formula with the example of $P_{29}(16x+10)$ considered above. Let $z=16x+10$. From Table \ref{t2}, we observe that $P_{29}$ has five zeros $z_i$ with $z_i$ odd, four with $z_i\equiv0$ mod 4, one with $z_i\equiv6$ mod 8, one with $z_i\equiv2$ mod 16, and one with $z_i=16x_0+10$ for some $x_0\in\zt$. The sum of $\nu(z-z_i)$ is $5\cdot0+4\cdot1+2+3+\nu((16x+10)-(16x_0+10))=13+\nu(x-x_0)$. Since $\nu([\frac{29-1}2]!)=11$, (\ref{cgen}) becomes $\nu(P_{29}(z))=\nu(x-x_0)+2$, consistent with the worked-out example above.

\section{Some proofs}\label{pfsec}
In this section, we prove Proposition \ref{per}, Theorem \ref{four}, and Theorem \ref{double}.
We will need the following  results related to $\nu(\Phi_n(k))$. One such result was stated earlier as Proposition \ref{DSprop}.

The next result is an extension of \cite[Prop 2.4]{partial}. The proof of the new part, the condition for equality, will appear at the end of this section.
\begin{prop}\label{qprop1} For any nonnegative integers $n$ and $k$,
$$\nu(\Phi_n(k))\ge0.$$
If $n=2^e+\Delta$, with $0\le\Delta<2^e$, then equality occurs here iff $\binom{2^{e-1}-1-[\Delta/2]}{k-\Delta}$ is odd.
\end{prop}

The next result restates \cite[Props 2.5,2.6]{partial}.
\begin{prop}\label{refine}  Let $n$ and $k$ be nonnegative integers with $n>k$. Then $\nu(\Phi_n(k))\ge\a(n)-1-\a(k)$ with equality iff $\binom{n-1-k}k$ is odd.
Mod 4,
$$\sum_i\tbinom n{2i+1}i^k/(2^{n-1-2k}k!)\equiv\tbinom{n-1-k}k+\begin{cases}2\tbinom{n-1-k}{k-2}&\text{if $n-1$ and $k$ are even}\\
0&\text{otherwise.}\end{cases}$$
\end{prop}

The last result is an extension of \cite[Prop 2.3]{partial}. The proof of the new part  will appear near the end of this section.
\begin{prop}\label{stirowr} Mod $4$
$${\tfrac 1{n!}}\sum_i\tbinom{2n+\eps}{2i+b}i^k\equiv\begin{cases}S(k,n)+2nS(k,n-1)&\eps=0,\,b=0\\
(2n+1) S(k,n)+2(n+1)S(k,n-1)&\eps=1,\,b=0\\
S(k,n)+2(n+1)S(k,n-1)&\eps=1,\,b=1.
\end{cases}$$
Integrally
$${\tfrac 1{n!}}\sum_i\tbinom{2n}{2i+1}i^k=\sum_{d\ge0}2^{d+1}\tbinom{n+d}dS(k,n-1-d)/(2d+1)!!,$$
where $(2d+1)!!=\ds\prod_{i\le d}(2i+1)$.
\end{prop}

Next we prove Proposition \ref{per}.
\begin{proof}[Proof of Proposition \ref{per}]
Let $e=\lg(n)$. We have
$$P_n(x+2^t)-P_n(x)={\tfrac1{n!}}\sum{\tbinom n{2i+1}}(2i+1)^x((2i+1)^{2^t}-1)=\sum_{k\ge0}\sum_{j>0}T_{k,j},$$
where
$T_{k,j}={\tbinom xk\tbinom{2^t}j}\Phi_n(k+j)$,
By Proposition \ref{qprop1},  $\nu(T_{k,j})\ge t-e+1$ if $j<2^e$, while if $j\ge2^e$,
by Proposition \ref{DSprop} we obtain
$$\nu(T_{k,j})\ge t-\nu(j)+j-[n/2]\ge t+2^e-e-[n/2]\ge t-e+1.$$\end{proof}

Now we present the proof of Theorem \ref{four}.
\begin{proof}[Proof of Theorem \ref{four}] Let $e\ge2$, $1\le\Delta\le4$, $0\le p<2^{e-1}$, $p_2$ its mod 2 reduction, and
$$c=\begin{cases}2&\text{if $(\Delta,p_2)=(3,0)$ or $(4,1)$, $e=2$}\\
1&\text{if $(\Delta,p_2)=(3,0)$ or $(4,1)$, $e>2$}\\
0&\text{otherwise.}\end{cases}$$
By Corollary \ref{D2cor}, the theorem will follow from $\nu(P_{2^e+\Delta}(p))\ge c$, and, for $d\ge0$,
\begin{equation}\label{WTS}\nu(P_{2^e+\Delta}(2^{e-1}(x+2^d)+p)-P_{2^e+\Delta}(2^{e-1}x+p))-d=c.\end{equation}

Proposition \ref{qprop1} implies $\nu(P_{2^e+\Delta}(p))\ge0$ and $P_{2^e+\Delta}(p)\equiv\sum\binom p k\binom{2^{e-1}-1-[\Delta/2]}{k-\Delta}$ mod 2,
which is easily seen to be 0 mod 2 if $(\Delta,p_2)=(3,0)$ or $(4,1)$. Showing $\nu(P_{2^e+\Delta}(p))\ge2$ when $e=2$ and $(\Delta,p_2)=(3,0)$ or $(4,1)$ is accomplished
by using \ref{DSprop} to eliminate all but some very small values of $k$ and checking these by direct computation.

Now we prove (\ref{WTS}).
The LHS  equals $\ds\sum_{k\ge0,\ j>0}T_{k,j}$, where
\begin{equation}\label{Tkj}T_{k,j}=\tbinom{2^{e-1}x+p}k\tfrac1{2^d}\tbinom{2^{d+e-1}}j\Phi_{2^e+\Delta}(j+k).\end{equation}

We first consider the case $e\ge3$.
 We will prove the six statements in Table \ref{t4}
 together with Claim \ref{cl1} and Claim \ref{cl2}, in which we assume $e\ge3$.

 \begin{diag}{Conclusions about $\nu(T_{k,j})$ when $e\ge3$}
\label{t4}
\begin{center}
\begin{tabular}{cc|cl}
$\Delta$&$p$&$\nu(T_{k,j})$&\quad equality iff\\
\hline
$1,2$&any&$\ge0$&$j=2^{e-1},\,k=0$\\
$3$&odd&$\ge0$&$j=2^{e-1},\,k=1$\\
$3$&$0\ (4)$&$\ge1$&$j=2^{e-1},\,k=0$\\
$3$&$2\ (4)$&$\ge1$&$j=2^{e-1},\,k=0,1,2$\\
$4$&even&$\ge0$&$j=2^{e-1},\,k=0$\\
$4$&odd&$\ge0$&$j=2^{e-1},\,k=0,1$\\
\hline
\end{tabular}
\end{center}
\end{diag}

\begin{claim}\label{cl1} If $\Delta=4$ and $p$ is odd, then $\nu(T_{k,j})=1$ iff $j=2^{e-1}$, $k=2$, and $p\equiv3\ (4)$, or $j=2^{e-2}$, $k\equiv0,1\ (4)$, and $\binom p k$ odd.
\end{claim}
\begin{claim}\label{cl2} $\nu\bigl(\sum\binom{2^e+4}{2i+1}(2i)^{2^{e-1}}\bigr)=\nu\bigl(\sum\binom{2^e+4}{2i+1}(2i)^{2^{e-1}+1}\bigr)$ and their odd factors are both $\equiv3\ (4)$.
\end{claim}

One easily checks that this implies the result. For example, if $\Delta=3$ and $p$ is even, the table says that $\nu(\sum T_{k,j})$ is determined by exactly 1 or 3 terms having $\nu=1$.
The only place where a little argument is required is the case $\Delta=4$, $p$ odd. In this case, mod $4$, we get
$(p+1)$ from $j=2^{e-1}$, $k=0,1$, and also get $2$ if $p\equiv3\ (4)$. The sum of these is $2$.
Here we have used Claims \ref{cl2} and \ref{cl1} and the fact that if $p$ is odd and $k\equiv0$ mod 4, then $\binom p k\equiv\binom p {k+1}$ mod 2.

We have $$V:=\nu(T_{j,k})=\nu\tbinom p k+e-1-\nu(j)+\nu(\Phi_{2^e+\Delta}(j+k)).$$ This is the quantity in the third column of Table \ref{t4}.
Now we verify the claims of Table 1 and Claims \ref{cl1} and \ref{cl2}.

By Proposition \ref{DSprop},  if $\nu(j)>e-1$ or $j=3\cdot 2^{e-2}$, then $V>\!>0$, and so we need not consider
these cases when looking for $V=0$ or $V=1$. By Proposition \ref{qprop1},  we deduce that $V\ge0$ with equality iff $j=2^{e-1}$, $\binom p k$ odd,
and $\binom{2^{e-1}-1-[\Delta/2]}{2^{e-1}+k-\Delta}$ odd. The latter condition is equivalent to $\binom{\Delta-1-k}{[\Delta/2]}$ odd, which, incorporating also $\binom p k$ odd, comprises the pairs $(\Delta,k)=(1,0)$,
$(2,0)$, $(3,1)$, $(4,0)$, and $(4,1)$. This implies all of our claims regarding when $V=0$.

Now let $(\Delta,p_2)=(3,0)$. If $j=2^{e-2}$, then $V=1$ iff $\binom pk$ is odd and (using Proposition \ref{qprop1}) $\binom{2^{e-1}-2}{2^{e-2}+k-3}$ is odd, which
is impossible. If $j=2^{e-1}$, then $V=1$ if $\nu\binom p k=1$ and $k=1$ or if $\binom p k$ is odd and $\Phi_{2^e+3}(2^{e-1}+k)\equiv2\ (4)$.
By Proposition \ref{refine}, $\Phi_{2^e+3}(2^{e-1})\equiv2\ (4)$, and by Proposition \ref{stirowr} \begin{equation}\label{34}\Phi_{2^e+3}(2^{e-1}+k)\equiv\begin{cases}2&k=2\\ 0&k>2\end{cases} \mod 4,\end{equation}
implying our claims in this case.

Finally let $(\Delta,p_2)=(4,1)$. Claim \ref{cl2} is proved using Proposition \ref{qprop1}, and \ref{refine} for the first sum and \ref{stirowr} for the second.
If $j=2^{e-2}$, then $V=1$ iff $\binom p k$ and $\binom{2^{e-1}-3}{2^{e-2}+k-4}$ are odd, and this happens in the asserted situations.
 If $j=2^{e-1}$, then $V=1$ iff $\binom p k\equiv2\ (4)$ and $k=0,1$, which is impossible (since $p$ is odd), or if
 $\binom p k$ is odd and $\Phi_{2^e+4}(2^{e-1}+k)\equiv2\ (4)$. Since (\ref{34}) is also true for $\Phi_{2^e+4}$ by \ref{stirowr}, Claim \ref{cl1} is clear.

 The main difference when $e=2$ is that $j=2^e$ can play a significant role. When $e>2$, Proposition \ref{DSprop} implied that $\nu(\Phi_{2^e+\Delta}(j+k))$ would be too large to have an effect.

 The case $e=2$ involves $P_5$, $P_6$, $P_7$, and $P_8$. The result for them was part of \cite[Theorem 2.1]{partial}, although detailed proofs were
 not presented there for $P_7$ and $P_8$. As when $e>2$, the argument considers the terms $T_{k,j}$ of (\ref{Tkj}).  We illustrate with the case
 $e=2$, $\Delta=3$, $p=0$, $d\ge1$. We wish to prove that $\nu(\sum T_{k,j})=2$, where
 $T_{k,j}=\tbinom{2x}k\frac1{2^d}\tbinom{2^{d+1}}j\Phi_7(j+k)$,
 $k\ge0$, $j>0$. We can use Proposition \ref{DSprop} to eliminate large values of $j+k$. If $x\equiv0$ mod 4, the only terms that are nonzero mod $8$ occur when $k=0$ and $j=1$, $3$, and $4$.
 These three terms have 2-exponents $2$, $1$, and $1$, respectively, and the sum of the last two is divisible by $8$.

 If $x\equiv2$ mod 4, or $x$ odd, several additional terms are involved, but the same conclusion is obtained.
\end{proof}

\begin{proof}[Proof of Theorem \ref{double}]

The inequality in (\ref{alph}) follows easily from Proposition \ref{refine} and \cite[Lemma 2.40]{partial} (and Proposition \ref{DSprop} to handle $j=2^e$). These results also imply that, if $\eps=0=\delta$, equality is obtained in (\ref{alph})
iff $k=p$, $\binom{n-1-j-k}{j+k}$ is odd, and $j=2^e-2^h$ with $2^h>k$. Thus the case $\eps=0=\delta$ of the theorem follows from the following lemma.
\begin{lem} If $\binom{n-1-p}p$ is odd, $e=\lg(n-1)$, $p<2^h<2^e$, and $p<[(n-3\cdot2^{e-1})/2]$, then
$$\binom{n-1-p-2^e+2^h}{p+2^e-2^h}\text{ is odd iff }h=\lg(2^{e+1}-n+p).$$\label{lemlem}
\end{lem}
\begin{proof} Let $\ell=\lg(2^{e+1}-n+p)$, $A=n-1-p-2^e+2^h$, and $B=p+2^e-2^h$. If $h<\ell$, then $2^{h+1}\le2^{e+1}-n+p$, from which is follows that $A<B$, and hence $\binom AB=0$.

If $h>\ell$, it follows that $2^e\le A<3\cdot2^{e-1}$ and $2^{e-1}\le B<2^e$, so $\binom AB$ is even, due to the $2^{e-1}$ position.  If $h=\ell$, it is immediate that $n-1-p=2(2^e-2^h)+L$ with $0\le L<2^h$ and $\binom Lp$ odd. This implies that $\binom AB=\binom{2^e-2^h+L}{2^e-2^h+p}$ is odd.
\end{proof}

If $\eps=1$ and $\delta=0$, the above methods together with \cite[Lemma 2.44]{partial} show that equality in (\ref{alph}) is obtained only for $(j,k)=(2^e-2^\ell,p)$, as claimed.

Now let $\delta=1$ and $\eps=0$. If $k\le p$, the $\delta=0$ analysis applies to give the $(j,k)=(2^e-2^\ell,p)$ solution.
If $k=2^{e-1}+\Delta$ with $\Delta\ge0$, then analysis similar to that performed above implies that equality is obtained in (\ref{alph}) iff
$\Delta=p$, $\phi(j,2^{e-1}+\Delta)=e-1$, where $\phi$ is as in \cite[Lemma 2.40]{partial}, and $\binom{n-1-j-2^{e-1}-\Delta}{j+2^{e-1}+\Delta}$ is odd with $2(j+2^{e-1}+\Delta)<n$.
Part 3 of \cite[Lemma 2.44]{partial} (with its $e-1$ corresponding to our $e$) gives two possibilities for $\phi=e-1$.  The first one does not satisfy $2(j+2^{e-1}+\Delta)<n$.
The second one reduces to $j=2^{e-1}-2^h$ with $\Delta<2^h$. Since $\Delta=p$, the required oddness of the binomial coefficient becomes exactly the condition of Lemma \ref{lemlem},
and so we obtain that $h$ equals the $\ell$ of our Theorem \ref{double}.

If $p<k<2^{e-1}$, the condition for equality in (\ref{alph}) becomes $\a(2^{e-1}+p-k)=1$, $\binom{n-1-j-k}{j+k}$ odd with $2(j+k)<n$, and $\phi(j,k)=e$. The first of these says $k=2^{e-1}+p-2^h$ with $h<e-1$. By part 2 of \cite[Lemma 2.44]{partial}, $j=2^e-2^t$ with $2^t>2^{e-1}+p-2^h$, which implies $t=e-1$ since $h<e-1$. Thus $j=2^{e-1}$, and the odd binomial coefficient is again handled by Lemma \ref{lemlem}, implying that $h$ equals the $\ell$ of the theorem.

The case $\delta=1=\eps$ is established using the same methods.
\end{proof}

\begin{proof}[Proof of last part of Proposition \ref{stirowr}] We extend the proof in \cite{partial}.  We have
$$ \sum_i\tbinom{2n}{2i+1}i^k
=\sum_\ell C_{2n,\ell,1} \ell! S(k,\ell),$$
where $C_{2n,\ell,1}={\displaystyle\sum\limits_{i}\tbinom {2n}{2i+1}\tbinom i\ell}$.
We will prove the possibly new result
\begin{equation}\label{new?}\sum\tbinom{2n}{2i+1}\tbinom i{n-d-1}=2^{2d+1}\tbinom{n+d}{2d+1}.\end{equation}
Then, using (\ref{new?}) at the second step,
\begin{eqnarray*}{\tfrac1{n!}}\sum\tbinom{2n}{2i+1}i^k&=&{\tfrac1{n!}}\sum_{d\ge0}C_{2n,n-d-1,1}(n-d-1)!S(k,n-d-1)\\
&=&{\tfrac1{n!}}\sum_d2^{2d+1}\tbinom{n+d}{2d+1}(n-d-1)!S(k,n-d-1)\\
&=&\sum_d2^{2d+1}\tfrac{(n+d)!}{n!(2d+1)!}S(k,n-d-1),\end{eqnarray*}
and the result follows since $\frac{d!}{(2d+1)!}=1/(2^d(2d+1)!!)$.

We prove (\ref{new?}) with help from \cite{AB} and the associated software. Let $d$ be fixed, and
$$F(n,i)=\frac{\binom{2n}{2i+1}\binom i{n-d-1}}{2^{2d+1}\binom{n+d}{2d+1}}.$$
We will show that \begin{equation}\label{sumseq}\sum_iF(n+1,i)=\sum_iF(n,i).\end{equation}
Since $\sum_iF(d+1,i)=1$, this implies that $\sum_iF(n,i)=1$ for all $n$, our desired result.

To prove (\ref{sumseq}), let
$$G(n,i)=\frac{(2n+1-i)(2i+1)(i+d+1-n)}{(d+n+1)(n-i)(-2n+2i-1)}F(n,i).$$
(This is what was discovered by the software.) Then one can verify
$$F(n+1,i)-F(n,i)=G(n,i+1)-G(n,i).$$
When summed over $i$, the RHS equals 0, implying (\ref{sumseq}).
\end{proof}

\begin{proof}[Proof of second part of Proposition \ref{qprop1}] We want to know when $\Phi_{2^e+\Delta}(k)$ is odd.
We expand $(2i)^k$ as $\sum(-1)^j\binom kj(2i+1)^j$. Using Proposition \ref{P0}, we are reduced to proving, when $\Delta=2d+1$,
\begin{equation}\label{bcs}\binom{2^{e-1}-1-d}{k-2d-1}\equiv\sum\binom kj\binom{j-2^{e-1}-d-1}{2^{e-1}+d}\mod 2\end{equation}
and a similar result when $\Delta=2d$. If $k<2^{e-1}+d+1$, the RHS equals
$$\sum[x^j](1+x)^k\cdot[x^{2^{e-1}+d-j}](1+x)^{-2^{e-1}-d-1}=\binom{-2^{e-1}-d-1+k}{2^{e-1}+d}.$$
Thus both sides of (\ref{bcs}) are odd iff the binary expansions of $d$ and $k-2d-1$ never have 1's in the same position.
If $k\ge2^{e-1}+d+1$, the LHS of (\ref{bcs}) is 0. To evaluate the RHS, note that $\binom{j-2^{e-1}-d-1}{2^{e-1}+d}\equiv\binom{j-2^eA-2^{e-1}-d-1}{2^{e-1}+d}$ mod 2.
Choose $2^eA$ so that $k-2^eA-2^{e-1}-d-1<0$. Then the RHS becomes $\binom{-2^{e-1}-d-1+k}{2^eA+2^{e-1}+d}=0$.\end{proof}

\section{Proof of Theorem \ref{Delthm}}\label{limsec}
The proof of Theorem \ref{Delthm} will be aided by two lemmas.

\begin{lem}\label{lem1} For $0<d<2^e$,
$$\frac{2^{2^e+d-1-\a(d)}}{d!2^e!}\equiv \frac{2^{2^e+d-1-\a(d)}}{(2^e+d)!}\mod 2^{e-\lg(d)}.$$
\end{lem}
\begin{proof} First note that $\nu(\frac11+\frac12+\cdots+\frac1d)=-\lg(d)$, as is easily proved by induction on $d$. From this, we obtain,
$$\frac{(2^e+d)!/2^e!-d!}{d!}=\sum_{j\ge1}2^{je}\sigma_j(\tfrac11,\ldots,\tfrac1d)
\equiv0\mod 2^{e-\lg(d)},$$
where $\sigma_j$ is the elementary symmetric polynomial. The terms with $j\ge2$ are easily seen to have 2-exponent larger than that with $j=1$
by consideration of the largest 2-exponent in the denominator of any term of $\sigma_j$.
Multiplying the above by the odd number $2^{2^e+d-1-\a(d)}/(2^e+d)!$ yields the claim of the lemma.\end{proof}

 \begin{lem}\label{lem2} If $0<r\le D$, then $\nu(2^{2^e-r-1}/(2^e-r)!)\ge e-1-\lg(D)$.\end{lem}
 \begin{proof} The indicated exponent equals $\a(2^e-r)-1=e-\a(r-1)-1$, and $\a(r-1)\le\lg(D)$.\end{proof}

\begin{proof}[Proof of Theorem \ref{Delthm}] We have
\begin{eqnarray*}P_{2^e+\Delta}(x)&=&\sum_{i\ge0}\frac1{(2i+1)!(2^e+\Delta-2i-1)!}\sum_{k=0}^xS(x,k)k!\tbinom{2i+1}k\\
&=&\sum_{k=0}^xS(x,k)\frac1{(2^e+\Delta-k)!}\sum_{i\ge0}\binom{2^e+\Delta-k}{2i+1-k}\\
&=&\sum_{k=0}^xS(x,k)\frac1{(2^e+\Delta-k)!}2^{2^e+\Delta-k-1}\\
&\equiv&\sum_{k=0}^\Delta S(x,k)\frac1{(2^e+\Delta-k)!}2^{2^e+\Delta-k-1}\mod 2^{e-\lg(x-\Delta)-1}\\
&\equiv&\sum_{k=0}^\Delta S(x,k)\frac1{2^e!(\Delta-k)!}2^{2^e+\Delta-k-1}\mod 2^{e-\lg(\Delta)+1}\\
&=&\frac1{\U(2^e!)\Delta!}\sum_{k=0}^\Delta S(x,k)k!2^{\Delta-k}\tbinom \Delta k\\
&=&\frac1{\U(2^e!)\Delta!}\sum_{k=0}^\Delta\sum_j(-1)^{k+j}\tbinom kjj^x2^{\Delta-k}\tbinom{\Delta}k\\
&=&\frac1{\U(2^e!)\Delta!}\sum_jj^x\tbinom{\Delta}j.\end{eqnarray*}
Lemmas \ref{lem2} and \ref{lem1} are used to prove the two congruences.
Equality occurs at the first congruence if $\Delta\ge x$ since $S(x,k)=0$ when $k>x$.
To see the last step, let $\ell=\Delta-k$ and obtain
$$\sum_{k=0}^\Delta(-1)^{k+j}\tbinom kj 2^{\Delta-k}\tbinom{\Delta}k=\tbinom\Delta j(-1)^{\Delta-j}\sum_{\ell=0}^\Delta(-2)^\ell\tbinom{\Delta-j}\ell=\tbinom{\Delta}j.$$
\end{proof}

\def\line{\rule{.6in}{.6pt}}

\end{document}